\documentclass[11pt,a4paper,oneside]{amsart}

\newcounter{commentcounter}

\usepackage{amsmath} 
\usepackage{amssymb} 
\usepackage{amsthm} 
\usepackage{stmaryrd} 
\usepackage[english]{babel} 
\usepackage[font=small,justification=centering]{caption} 
\usepackage[nodayofweek]{datetime}
\usepackage[shortlabels]{enumitem} 
\usepackage[T1]{fontenc} 
\usepackage[utf8]{inputenc} 
\usepackage{ifthen} 
\usepackage{mathabx} 
\usepackage{mathtools} 
\usepackage[dvipsnames]{xcolor} 
\usepackage[pdftex,  colorlinks=true, pagebackref=true]{hyperref} 
    \hypersetup{urlcolor=RoyalBlue, linkcolor=RoyalBlue,  citecolor=black}
\usepackage{setspace} 
\usepackage{tikz-cd} 
\usepackage{xfrac} 
\usepackage[capitalize]{cleveref} 
\usepackage{booktabs}
\usepackage{multirow,multicol}
\renewcommand*{\backref}[1]{}
\renewcommand*{\backrefalt}[4]
{
    \ifcase #1
        No citation in the text.
    \or
        Cited on Page #2.
    \else
        Cited on Pages #2.
    \fi
}

\makeatletter
\@namedef{subjclassname@1991}{Mathematical subject classification 1991}
\@namedef{subjclassname@2000}{Mathematical subject classification 2000}
\@namedef{subjclassname@2010}{Mathematical subject classification 2010}
\@namedef{subjclassname@2020}{Mathematical subject classification 2020}
\makeatother

\newtheorem{thm}{Theorem}[section]
\newtheorem{lemma}[thm]{Lemma}

\newtheorem{prop}[thm]{Proposition}

\newtheorem*{lemma*}{Lemma}

\theoremstyle{definition}
\newtheorem{defn}[thm]{Definition}

\theoremstyle{plain}

    \newtheoremstyle{TheoremNum}
        {8pt}{8pt} 
        {\itshape} 
        {-0.15cm} 
        {\bfseries} 
        {.} 
        { }  
        {\thmname{#1}\thmnote{ \bfseries #3}}
    \theoremstyle{TheoremNum}
    \newtheorem{duplicate}{}

\newcommand*{\claimproofname}{My proof}

\usepackage{tikz}
\usetikzlibrary{arrows,quotes}
\usetikzlibrary{shapes.geometric}
\tikzstyle{blackNode}=[fill=black, draw=black, shape=circle]

\title{On the derived length of Dyer groups}

\author{Olga Varghese}
\address{Olga Varghese\\ Institute of Mathematics, University of Münster, Einsteinstra{\upshape{\ss}}e 62, 48149, Münster, Germany}
\email{olga.varghese@uni-muenster.de}

\begin{document}

\begin{abstract}
      By definition, a group $G$ is quasi-perfect, if $G$ is perfect or the commutator subgroup of $G$ is perfect. In this note we give a description of quasi-perfect Dyer groups by properties
      of the corresponding Dyer graphs. 
\end{abstract}

\maketitle

\section{Introduction}
 The main objects in this article are groups that are defined by graphs. We start by recalling classical constructions of groups from graphs. Given a finite simplicial graph $\Gamma$ with vertex set $V(\Gamma)$, edge set $E(\Gamma)$ and with an edge-labelling $m\colon E(\Gamma)\rightarrow\mathbb{N}_{\geq 2}$,  the \emph{Coxeter group} $W_\Gamma$ associated to $\Gamma$ is the group with the presentation 
\begin{gather*}
W_\Gamma:=\left\langle V(\Gamma)\middle\vert 
\begin{array}{l} 
v^2=1\text{ for all }v\in V(\Gamma)\text{ and }\\ (vw)^{m(\left\{v,w\right\})}=1  \text{ if } \left\{v,w\right\}\in E(\Gamma)
\end{array}\right\rangle.
\end{gather*}
Another class of groups which is defined in a combinatorial way is the class of \emph{numbered graph products/graph products of cyclic groups}. Given non-trivial cyclic groups $G_1,\ldots, G_n$, the free product construction $G_1*G_2*\ldots*G_n$ is one tool to obtain a new group out of the given groups. Numbered graph products generalize this concept by constructing new groups out of vertex-labelled finite simplicial graphs where the vertices are labelled by $|G_1|,\ldots,|G_n|\in \mathbb{N}_{\geq 2}\cup \{\infty\}$. Given a finite simplicial graph $\Gamma$ and a map $f\colon V(\Gamma)\to\mathbb{N}_{\geq 2}\cup\left\{\infty\right\}$, the \emph{numbered graph product/graph product of cyclic groups} $G_\Gamma$ is defined as follows
\begin{gather*}
G_\Gamma:=\left\langle V(\Gamma)\middle\vert
\begin{array}{l}
v^{f(v)}=1\text{ if }f(v)<\infty\text{ and }\\ vwv^{-1}w^{-1}=1\text{ whenever }\left\{v, w\right\}\in E(\Gamma)
\end{array}\right\rangle.
\end{gather*}
These groups were introduced by Baudisch in \cite{Baudisch1981} for vertices with label $\infty$ and later by Green in \cite{Green1990} for vertices with arbitrary label.

Coxeter groups and numbered graph products are well-studied objects in geometric group theory. These two families of groups are usually studied separately in the literature. We study these groups in a unified way using the Dyer group framework.
 
Let $(\Gamma, m, f)$ be a triple consisting of a finite simplicial graph $\Gamma$, an edge-labelling $m\colon E(\Gamma)\to\mathbb{N}_{\geq 2}$, and a vertex-labelling $f\colon V(\Gamma)\to\mathbb{N}_{\geq 2}\cup\left\{\infty\right\}$. We call this triple a \emph{Dyer graph} if, for every edge $e=\left\{v,w\right\}\in E(\Gamma)$ with $m(e)\neq 2$, we have $f(v)=f(w)=2$. For two letters $a, b$ and $n\in\mathbb{N}$ we define $\pi(a,b,n):= abababa\ldots$ where the word has $n$ letters. The associated \emph{Dyer group} is defined as follows
\begin{gather*}
D_\Gamma:=\left\langle V(\Gamma)\middle\vert\begin{array}{l}
 v^{f(v)}=1 \text{ for } v\in V(\Gamma)\text{ if }f(v)\neq\infty,\\ \pi(v,w,m(\left\{v,w\right\}))=\pi(w,v,m(\left\{v,w\right\}))
\text{ if }\left\{v,w\right\}\in E(\Gamma)\end{array}\right\rangle\,.
\end{gather*}
We note that, if $f(v)=2$ for all $v\in V(\Gamma)$, then $D_\Gamma$ is a Coxeter group, and if $m(e)=2$ for all $e\in E(\Gamma)$, then $D_\Gamma$ is a numbered graph product. 

A recurring theme in the study of Dyer groups is that it is often possible to characterize properties of the group via the combinatorial structure of the graph. In this note our focus is on quasi-perfectness of Dyer groups. 

Let $G$ be a group and 
$[G,G]$ be the commutator subgroup of $G$. We define $G^{(0)}:=G$ and $G^{(i+1)}:=[G^{(i)}, G^{(i)}]$ for $i\geq 0$. 
The \emph{derived length} of $G$ is defined as ${\rm dl}(G)=\infty$ if $G^{(i)}\neq G^{(i+1)}$ for all $i\geq 0$; otherwise 
${\rm dl}(G):={\rm min}\left\{i\mid G^{(i)}=G^{(i+1)}\right\}$. 
By definition, a group $G$ is \emph{perfect} if ${\rm dl}(G)=0$ and $G$ is called \emph{quasi-perfect} if ${\rm dl}(G)\in\left\{0,1\right\}$. 
We note that given groups $G$ and $H$ we have ${\rm dl}(G\times H)={\rm max}\left\{{\rm dl}(G), {\rm dl}(H)\right\}$. 

Our first result gives a characterization of the derived length  of numbered graph products.

\medskip
\begin{duplicate}[Theorem \ref{QuasiPerfectGraphProducts}]
Let $G_\Gamma$ be a graph product of non-trivial cyclic groups. Then ${\rm dl}(G_\Gamma)\in\left\{1,2,\infty\right\}$. In particular, if $\Gamma$ is not a join, then
\begin{enumerate}
\item ${\rm dl}(G_\Gamma)=1$ if and only if $V(\Gamma)=\left\{v_1\right\}$.
\item ${\rm dl}(G_\Gamma)=2$ if and only if $V(\Gamma)=\left\{v,w\right\}$, $E(\Gamma)=\emptyset$ and $f(v)=f(w)=2$.
\end{enumerate}
\end{duplicate}
\medskip

A characterization of quasi-perfect Coxeter groups in terms of graphs was proven in \cite[Theorem 5.2]{BrooksbankPiggott2012}. We show that quasi-perfectness of Dyer groups can also be characterized using Dyer graphs.  
Let $(\Gamma, m, f)$ be a Dyer graph and $D_\Gamma$ be the associated Dyer group.
For every subset $T\subseteq V(\Gamma)$, the subgroup $D_T$ generated by $T$ is isomorphic to the Dyer group $D_\Delta$ where $\Delta$ is the subgraph of $\Gamma$ induced by the vertex set $T$, see \cite{Dyer1990}. A Dyer group $D_\Gamma$ is called \emph{even} if $E(\Gamma)=\emptyset$ or $m(E(\Gamma))\subseteq 2\mathbb{N}$. Using canonical retractions onto subgroups of type $D_T$ we show that the only quasi-perfect even Dyer groups are the abelian ones.

\medskip
\begin{duplicate}[Proposition \ref{evenDyerGroups}]
Let $D_\Gamma$ be an even Dyer group. Then $D_\Gamma$ is quasi-perfect if and only if $\Gamma$ is complete and $E(\Gamma)=\emptyset$ or $m(E(\Gamma))=\left\{2\right\}$.
\end{duplicate}
\medskip

Let $(\Gamma, m, f)$ be a Dyer graph. For a prime number $p$ let $\Gamma^p$ be the graph obtained from $\Gamma$ by removing all edges whose labels are divisible by $p$. 

\medskip
\begin{duplicate}[Theorem \ref{QuasiPerfectDyerGroups}]
Let $(\Gamma, m, f)$ be a Dyer graph and $D_\Gamma$ be the associated Dyer group. Let $V_1,\ldots, V_k$ be the vertex sets of the connected components of $\Gamma^2$ and $\Delta_1,\ldots,\Delta_k\subseteq\Gamma$ be the induced subgraphs. 

Then $D_\Gamma$ is quasi-perfect if and only if the following hold:
\begin{enumerate}
\item For $i\in\left\{1,\ldots, k\right\}$ and for every prime $p$ the graph $\Delta^p_i$ is connected.
\item For each pair $(i,j)$, $1\leq i<j\leq k$ there exist vertices $v\in V_i$ and $w\in V_j$ such that $m(\left\{v,w\right\})=2$. 
\end{enumerate}
\end{duplicate}
\medskip
Our proof relies on the description of quasi-perfect Coxeter groups \cite[Theorem 5.2]{BrooksbankPiggott2012} and on several types of quotients operations on Dyer groups. In particular, our strategy is to associate to a Dyer group $D_\Gamma$ an even Dyer group $D_\Omega$ where we can control ${\rm dl}(D_\Gamma)$ using ${\rm dl}(D_\Omega)$. 

\subsection*{Acknowledgements}
The author would like to thank the Isaac Newton Institute for Mathematical Sciences, Cambridge, for support and hospitali-\\ty during the programme Discrete and profinite groups, where work on this paper was undertaken. This work was supported by EPSRC grant EP/Z000580/1. The author is also funded by the Deutsche Forschungsgemeinschaft (DFG, German Research Foundation) under Germany's Excellence Strategy EXC 2044/2 –390685587, Mathematics Münster: Dynamics-Geometry-Structure.

\section{Simplicial graphs}

A (simplicial) \emph{graph} $\Gamma$ is a tuple $(V, E)$ where
$V$ is a non-empty set and $E$ is a subset of $\left\{\left\{v,w\right\}\mid v,w\in V, v\neq w\right\}$. A graph $\Gamma$ is called \emph{discrete} if $E(\Gamma)=\emptyset$. If $V'\subseteq V$ and $E'\subseteq E$ and $E'$ is a set of 2-element subsets of $V'$, then $\Gamma'=(V', E')$ is called a \emph{subgraph} of $\Gamma$. If $\Gamma'$ is a subgraph of $\Gamma$ and $E'$ contains all the edges $\left\{v, w\right\}\in E$ with $v, w\in V'$, then $\Gamma'$ is called an \emph{induced subgraph} of $\Gamma$. 

\begin{defn}
Let $\Gamma=(V,E)$ be a graph. The graph $\Gamma$ is a \emph{join} of two graphs $\Omega_1=(V_1, E_1)$ and $\Omega_2=(V_2, E_2)$ if $\Omega_1$  and $\Omega_2$ are induced subgraphs of $\Gamma$, $V$ is the disjoint union of $V_1$ and $V_2$ and for every pair $(v,w)\in V_1\times V_2$ we have $\left\{v,w\right\}\in E$. If $\Gamma$ is not a join, then we call $\Gamma$ \emph{indecomposable}.
\end{defn}

\begin{prop}
\label{SubgraphsIndecomposableGraph}
Let $\Gamma=(V,E)$ be an indecomposable graph. If $|V|\geq 3$, then $\Gamma$ has an induced subgraph isomorphic to $\Gamma_1$ or $\Gamma_2$ as shown in \Cref{fig:3verticies}. 

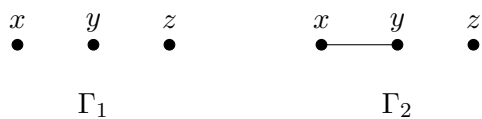
\begin{figure}[htb]
	\begin{center}
	\captionsetup{justification=centering}
		\begin{tikzpicture}
			\draw[fill=black]  (0,0) circle (2pt);
            \node at (0, 0.3) {$x$};
            \draw[fill=black]  (1,0) circle (2pt);
            \node at (1, 0.3) {$y$};
            \draw[fill=black]  (2,0) circle (2pt);
            \node at (2, 0.3) {$z$};
            \node at (1, -0.8) {$\Gamma_1$};

            \draw[fill=black]  (4,0) circle (2pt);
            \node at (4, 0.3) {$x$};
            \draw[fill=black]  (5,0) circle (2pt);
            \node at (5, 0.3) {$y$};
            \draw[fill=black]  (6,0) circle (2pt);
            \node at (6, 0.3) {$z$};
            \draw (4,0)--(5,0);
            \node at (5, -0.8) {$\Gamma_2$};           
\end{tikzpicture}
        \caption{Indecomposable graphs with $3$ vertices.}
	      \label{fig:3verticies}
    \end{center}
\end{figure}
\end{prop}

\begin{proof}

If $|V|=3$, then there are exactly $2$ indecomposable graphs with $3$ vertices as shown in 
\Cref{fig:3verticies}.

Let $|V|\geq 4$. Assume for contradiction that $\Gamma$ does not have an induced subgraph which is isomorphic to $\Gamma_1$ or $\Gamma_2$.
Then $\Gamma$ has an induced subgraph that is isomorphic to $\Gamma_3$ as shown in \Cref{fig:333verticies}. 

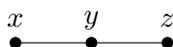
\begin{figure}[htb]
	\begin{center}
	\captionsetup{justification=centering}
		\begin{tikzpicture}

             \draw[fill=black]  (8,0) circle (2pt);
            \node at (8, 0.3) {$x$};
            \draw[fill=black]  (9,0) circle (2pt);
            \node at (9, 0.3) {$y$};
            \draw[fill=black]  (10,0) circle (2pt);
            \node at (10, 0.3) {$z$};
            \draw (8,0)--(9,0);
            \draw (9,0)--(10,0);
            
\end{tikzpicture}
        \caption{A graph $\Gamma_3$.}
	      \label{fig:333verticies}
    \end{center}
\end{figure}

Since $n\geq 4$, there exists a vertex $v\in V$ that is connected to the subgraph $\Gamma_3$, since otherwise $\Gamma$ would have an induced subgraph of type $\Gamma_2$. Let $\Omega$ be the induced subgraph of $\Gamma$ with the vertex set $V(\Gamma_3)\cup\left\{v\right\}$.  
Let us consider all connected graphs with $4$ vertices, see \Cref{fig:connectedgraphs4vertices}.
\begin{figure}[htb]
	\begin{center}
	\captionsetup{justification=centering}
		\begin{tikzpicture}
			\draw[fill=black]  (0,0) circle (2pt);
            \draw[fill=black]  (0,1) circle (2pt);
            \draw[fill=black]  (1,0) circle (2pt);
            \draw[fill=black]  (1,1) circle (2pt);
            \draw (0,0)--(1,0);
            \draw (0,1)--(0,0);
            \draw (1,0)--(1,1);
            \node at (0.5, -0.5) {$\Omega_1$};

            \draw[fill=black]  (3,0) circle (2pt);
            \draw[fill=black]  (3,1) circle (2pt);
            \draw[fill=black]  (4,0) circle (2pt);
            \draw[fill=black]  (4,1) circle (2pt);
            \draw (3,0)--(4,0);
            \draw (3,0)--(4,1);
            \draw (3,0)--(3,1);
            \node at (3.5, -0.5) {$\Omega_2$};

            \draw[fill=black]  (6,0) circle (2pt);
            \draw[fill=black]  (6,1) circle (2pt);
            \draw[fill=black]  (7,0) circle (2pt);
            \draw[fill=black]  (7,1) circle (2pt);
            \draw (6,0)--(6,1);
            \draw (6,0)--(7,0);
            \draw (7,0)--(7,1);
            \draw (6,1)--(7,1);
            \node at (6.5, -0.5) {$\Omega_3$};

            \draw[fill=black]  (0,-2) circle (2pt);
            \draw[fill=black]  (0,-3) circle (2pt);
            \draw[fill=black]  (1,-2) circle (2pt);
            \draw[fill=black]  (1,-3) circle (2pt);
            \draw (0,-3)--(1,-3);
            \draw (0,-3)--(1,-2);
            \draw (0,-3)--(0,-2);
            \draw (1,-3)--(1,-2);
            \node at (0.5, -3.5) {$\Omega_4$};

            \draw[fill=black]  (3,-2) circle (2pt);
            \draw[fill=black]  (3,-3) circle (2pt);
            \draw[fill=black]  (4,-2) circle (2pt);
            \draw[fill=black]  (4,-3) circle (2pt);
            \draw (3,-3)--(4,-3);
            \draw (3,-3)--(4,-2);
            \draw (3,-3)--(3,-2);
            \draw (3,-2)--(4,-2);
            \draw (4,-3)--(4,-2);
            \node at (3.5, -3.5) {$\Omega_5$};

            \draw[fill=black]  (6,-2) circle (2pt);
            \draw[fill=black]  (6,-3) circle (2pt);
            \draw[fill=black]  (7,-2) circle (2pt);
            \draw[fill=black]  (7,-3) circle (2pt);
            \draw (6,-3)--(7,-3);
            \draw (6,-3)--(7,-2);
            \draw (6,-3)--(6,-2);
            \draw (7,-3)--(7,-2);
            \draw (7,-3)--(6,-2);
            \draw (6,-2)--(7,-2);
            \node at (6.5, -3.5) {$\Omega_6$};

        \end{tikzpicture}
        \caption{Connected graphs with four vertices.}
	      \label{fig:connectedgraphs4vertices}
    \end{center}
\end{figure}
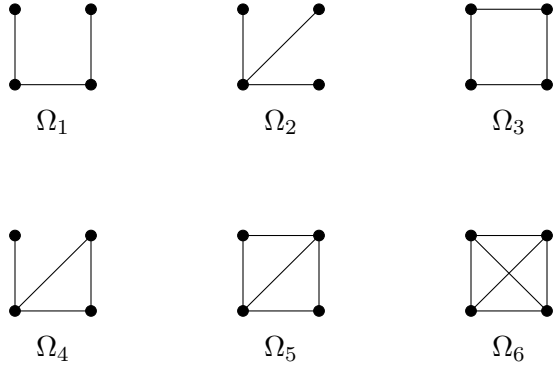

The graph $\Omega$ is connected, does not have an induced subgraph isomorphic to $\Gamma_1$ or $\Gamma_2$ and has an induced subgraph isomorphic to $\Gamma_3$. Thus $\Omega$ is isomorphic to $\Omega_3$ or $\Omega_5$.

If $\Omega\cong \Omega_3$, then there exists a vertex $w\in V$, $w\notin V(\Omega)$, since $\Omega_3$ is a join.
Now we consider the subgraph $\Delta$ induced by the vertex set $V(\Omega)\cup\left\{w\right\}$. The graph $\Delta$ is connected and every connected subgraph induced by $4$ vertices is isomorphic to $\Omega_3$ or $\Omega_5$ and this is not possible.

If $\Omega\cong \Omega_5$, then the same arguments as above lead to a contradiction.
Hence, $\Gamma$ has an induced subgraph that is isomorphic to $\Gamma_1$ or $\Gamma_2$.
\end{proof}

\section{The derived series of groups}
Let $G$ be a group and $[G,G]:=\langle\left\{ xyx^{-1}y^{-1}\mid x,y\in G\right\}\rangle$ be the \emph{commutator subgroup} of $G$. We define $$G^{(0)}:=G\text{ and }G^{(i+1)}:=[G^{(i)}, G^{(i)}]\text{ for }i\geq 0.$$ 
We write ${\rm dl}(G)=\infty$ if $G^{(i)}\neq G^{(i+1)}$ for all $i\geq 0$; otherwise 
$${\rm dl}(G):={\rm min}\left\{i\mid G^{(i)}=G^{(i+1)}\right\}.$$
The series
$G=G^{(0)}\supseteq G^{(1)}\supseteq G^{(2)}\supseteq\ldots$
is called the \emph{derived series} of $G$. We note that for $i\geq 0$ the subgroup $G^{(i)}$ is characteristic in $G$ and is therefore normal in $G$, see \cite[Theorem 5.21]{Rotman1995}.

We are interested in the image of the map ${\rm dl}$ for interesting classes of groups. For example, let $\mathcal{D}$ be the family consisting of all dihedral groups. By definition, a group $G$ is a \emph{dihedral group} if $G\cong \mathbb{Z}/2\mathbb{Z}*\mathbb{Z}/2\mathbb{Z}$ or if there exists $n\in\mathbb{N}_{\geq 2}$ such that 
$G\cong\langle x,y\mid x^2=1, y^2=1, (xy)^n=1\rangle$. It is easy to see that ${\rm dl}(\mathcal{D})=\left\{1,2\right\}$.

For $n\in\mathbb{N}$ we denote by $\mathcal{W}_n$ all Coxeter groups where the Coxeter graph has $n$ vertices.
In \cite{Jeong2006}, it was proven that ${\rm dl}(\mathcal{W}_3)=\left\{1,2,3,\infty\right\}$ and ${\rm dl}(\mathcal{W}_4)=\left\{1,2,3,4,\infty\right\}$. For the affine Coxeter group of type $\widetilde{D}_4$ we have ${\rm dl}(W_{\widetilde{D}_4})=5$, see \cite{EdjvetJeong2007}. Hence, $\left\{1,2,3,4,5,\infty\right\}\subseteq {\rm dl}(\mathcal{W}_5)$. It is not known if there exists a Coxeter group $W_\Gamma\in \mathcal{W}_5$ such that ${\rm dl}(W_\Gamma)\in\mathbb{N}_{\geq 6}$.

We now prove several lemmas. The first lemma shows that the map ${\rm dl}$ behaves very nicely to products of groups.
\begin{lemma}
\label{Products}
Let $G$ and $H$ be groups. Then ${\rm dl}(G\times H)={\rm max}\left\{{\rm dl}(G), {\rm dl}(H)\right\}$.
\end{lemma}
\begin{proof}
For $i\geq 0$ we have $(G\times H)^{(i)}=G^{(i)}\times H^{(i)}$. Hence, 
${\rm dl}(G\times H)={\rm max}\left\{{\rm dl}(G), {\rm dl}(H)\right\}$.
\end{proof}

\begin{lemma}
\label{Quasi_Perfect_Epi}
Let $G$ and  $H$ be groups and $\psi\colon G\to H$ be a homomorphism. If $\psi$ is surjective, then ${\rm dl}(G)\geq {\rm dl}(H)$.  
\end{lemma}
\begin{proof}
If ${\rm dl}(G)=\infty$, then there is nothing to prove. Thus assume that ${\rm dl}(G)=n$.
For $i\geq 0$ we have $\psi(G^{(i)})=\psi(G)^{(i)}=H^{(i)}$.
Since $G^{(n)}=G^{(n+1)}$, we obtain $H^{(n)}=H^{(n+1)}$. Hence $n\geq {\rm dl}(H)$.
\end{proof}

\begin{lemma}(\cite[p.14]{LyndonSchupp2001})
\label{FreeGroupinfinity}
Let $F_n$ be the free group of rank $n$. If $n\geq 2$, then ${\rm dl}(F_n)=\infty$.
In particular, if a free group $F$ has infinite rank, then ${\rm dl}(F)=\infty$.
\end{lemma}

\begin{lemma}
\label{DerivedSeriesFreeProduct}
Let $G$ and $H$ be non-trivial abelian groups. Then ${\rm dl}(G*H)\in\left\{2,\infty\right\}$. In particular, ${\rm dl}(G*H)=2$ if and only if $G\cong H\cong\mathbb{Z}/2\mathbb{Z}$.
\end{lemma}
\begin{proof}
The statement of the lemma follows from the fact that the commutator subgroup of $G*H$ is a free group of rank $(|G|-1)\cdot (|H|-1)$, see \cite[\S 1 Proposition 4]{Serre2003}. Since a non-abelian free group has infinite derived length by Lemma \ref{FreeGroupinfinity}, we obtain ${\rm dl}(G*H)=\infty$ if and only if $|G|\geq 3$ or $|H|\geq 3$.
\end{proof}

Let $G_\Gamma$ be a graph product of non-trivial cyclic groups. It is straightforward to verify that for every subset $T\subseteq V(\Gamma)$ the map $\rho'_T\colon V(\Gamma)\to G_T$ defined as follows: $v\mapsto v$ if $v\in T$ and $v\mapsto 1$ if $v\notin T$ induces a canonical  retraction $\rho_T\colon G_\Gamma\twoheadrightarrow G_T$.

\begin{thm}
\label{QuasiPerfectGraphProducts}
Let $G_\Gamma$ be a graph product of non-trivial cyclic groups. Then ${\rm dl}(G_\Gamma)\in\left\{1,2,\infty\right\}$. In particular, if $\Gamma$ is not a join, then
\begin{enumerate}
\item ${\rm dl}(G_\Gamma)=1$ if and only if $V(\Gamma)=\left\{v_1\right\}$. 
\item ${\rm dl}(G_\Gamma)=2$ if and only if $V(\Gamma)=\left\{v,w\right\}$, $E(\Gamma)=\emptyset$ and $f(v)=f(w)=2$.
\end{enumerate}
\end{thm}
\begin{proof}
Using the result of Lemma \ref{Products} we can assume that $\Gamma$ is decomposable. Let $V(\Gamma)=\left\{v_1,\ldots, v_n\right\}$. If $n=1$, then $G_\Gamma$ is cyclic and therefore ${\rm dl}(G_\Gamma)=1$.
If $n=2$, then $\Gamma$ is a discrete graph. By Lemma \ref{DerivedSeriesFreeProduct} follows that ${\rm dl}(G_\Gamma)\geq 2$. Moreover, ${\rm dl}(G_\Gamma)=2$ if and only if $f(v)=f(w)=2$, otherwise ${\rm dl}(G_\Gamma)=\infty$.

Let $n\geq 3$. By Proposition \ref{SubgraphsIndecomposableGraph} we know that $\Gamma$ has an induced subgraph $\Omega$ that is isomorphic to $\Gamma_1$ or $\Gamma_2$ as shown in \Cref{fig:3verticies}.
We consider the canonical projections 
$$G_\Gamma\twoheadrightarrow \langle x,y\rangle *\langle z\rangle\twoheadrightarrow (\langle x\rangle\times\langle y\rangle) * \langle z\rangle.$$ 
By Lemma \ref{Quasi_Perfect_Epi} and Lemma \ref{DerivedSeriesFreeProduct} follows that 
$${\rm dl}(G_\Gamma)\geq {\rm dl}((\langle x\rangle\times\langle y\rangle) * \langle z\rangle)=\infty.$$
\end{proof}

\subsection{Even Dyer groups}
Let $(\Gamma, m, f)$ be a Dyer graph and $D_\Gamma$ be the associated Dyer group. By definition, $D_\Gamma$ is called \emph{even} if $E(\Gamma)=\emptyset$ or $m(E(\Gamma))\subseteq 2\mathbb{N}$.
There are a number of techniques that can be applied more easily to even Dyer groups than to arbitrary Dyer groups. For instance, it is straightforward to verify that for every $T\subseteq V(\Gamma)$ there exists a retraction $\rho_T\colon D_\Gamma\twoheadrightarrow D_T$. The epimorphism $\rho_T$ is induced by the map $\rho_T'\colon V(\Gamma)\to D_T$ where $\rho_T'(v)=v$ if $v\in T$ and $\rho_T'(v)=1$ if $v\notin T$.

\begin{prop}
\label{evenDyerGroups}
Let $D_\Gamma$ be an even Dyer group. Then $D_\Gamma$ is quasi-perfect if and only if $\Gamma$ is complete and $E(\Gamma)=\emptyset$ or $m(E(\Gamma))=\left\{2\right\}$.
\end{prop}
\begin{proof}
Assume first that $\Gamma$ is complete and $E(\Gamma)=\emptyset$ or $m(E(\Gamma))=\left\{2\right\}$. It is easy to see that in this case $D_\Gamma$ is abelian and therefore ${\rm dl}(D_\Gamma)=1$.

For the other direction, assume for contradiction that there exist vertices $v,w\in V(\Gamma)$, $v\neq w$ such that $\left\{v,w\right\}\notin E(\Gamma)$. Then we have a retraction $D_\Gamma\twoheadrightarrow \langle v\rangle*\langle w\rangle$. By Lemma \ref{Quasi_Perfect_Epi} and Lemma \ref{DerivedSeriesFreeProduct} follows that ${\rm dl}(D_\Gamma)\geq 2$. Thus, if $D_\Gamma$ is quasi-perfect, then $\Gamma$ is complete. Furthermore, if $\Gamma$ has an edge $\left\{v,w\right\}$ with label $\geq 4$, then using the retraction $D_\Gamma\twoheadrightarrow \langle v,w\rangle$ we obtain ${\rm dl}(D_\Gamma)\geq{\rm dl}(\langle v,w\rangle)=2$. Hence, every edge label in $\Gamma$ is equal to $2$.
\end{proof}

\subsection{Quasi-perfect Dyer groups}
Let $(\Gamma, m, f)$ be a Dyer graph and $D_\Gamma$ be the associated Dyer group. For a prime number $p$ let $\Gamma^p$ be the graph obtained from $\Gamma$ by removing all edges whose labels are divisible by $p$.  Let $V_1,\ldots, V_k$ be the vertex sets of the connected components of $\Gamma^2$ and $\Delta_1,\ldots,\Delta_k\subseteq\Gamma$ be the induced subgraphs. 
In order to show Theorem \ref{QuasiPerfectDyerGroups} we need several types of quotient operations on a Dyer group $D_\Gamma$. 
\begin{enumerate}
\item We define $V_{\geq 3}:=\left\{v\in V(\Gamma)\mid f(v)\geq 3\right\}$. Note that $D_{V_{\geq 3}}$ is a graph product of cyclic groups. We have a canonical retraction 
$$\rho_{V_{\geq 3}}\colon D_\Gamma\twoheadrightarrow D_{V_{\geq 3}}.$$

\item For $i=1,\ldots, k$ we have also a canonical retraction $$\rho_{V_i}\colon D_\Gamma\twoheadrightarrow D_{V_i}.$$
\item We want to point  out that for $i=1,\ldots, k$ the subgroup $D_{\Delta_i}$ is cyclic or is a Coxeter group. The abelianization of a Coxeter group $W_\Delta$ is isomorphic to $(\mathbb{Z}/2\mathbb{Z})^l$ where $l$ is the number of connected components in $\Delta^2$, see \cite[Proposition 2.2]{MollerVarghese2023}. Hence, if $D_{\Delta_i}$ is a Coxeter group, then all elements in $V(\Delta_i)$ are conjugate and the abelianization of $D_{\Delta_i}$ is isomorphic to $\mathbb{Z}/2\mathbb{Z}$. 

For $i=1,\ldots, k$ we fix $w_i\in V_i$. For each pair $(i,j)\in\left\{1,\ldots, k\right\}\times\left\{1,\ldots, k\right\}$, $1\leq i<j\leq k$ let 
$$a_{i,j}:={\rm gcd}\left\{m(\left\{v_i, v_j\right\})\mid v_i\in V_i, v_j\in V_j \right\}.$$
We define a new Dyer graph $(\Omega, m_\Omega, f_\Omega)$ which is obtained from $\Gamma$ as follows:
the vertex set of $\Omega$ is equal to $\left\{w_1,\ldots, w_k\right\}$. We define $f_\Omega(w_i):=f(w_i)$ for $i=1,\ldots, k$.
Two vertices $w_i, w_j$ are connected by an edge with label $a_{i,j}$ if and only if $a_{i,j}\neq\infty$. In particular, $D_\Omega$ is an even Dyer group, see \Cref{fig:GammaandOmega}. 

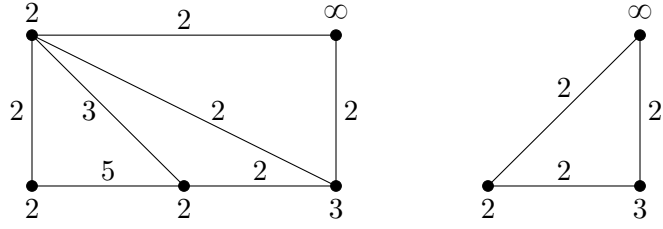
\begin{figure}[htb]
	\begin{center}
	\captionsetup{justification=centering}
        \begin{tikzpicture}
			\draw[fill=black]  (0,0) circle (2pt);
            \node at (0, -0.3) {$2$};
            \draw[fill=black]  (2,0) circle (2pt);
            \node at (2, -0.3) {$2$};
            \draw[fill=black]  (4,0) circle (2pt);
            \node at (4, -0.3) {$3$};
            \draw (0,0)--(2,0);
            \node at (1, 0.2) {$5$};
            \draw (0,0)--(0,2);
            \node at (-0.2, 1) {$2$};

            \draw (0,2)--(4,2);
            \node at (2, 2.2) {$2$};

            \draw (4,0)--(4,2);
            \node at (4.2, 1) {$2$};

            \draw (2,0)--(4,0);
            \node at (3, 0.2) {$2$};

            \draw (0,2)--(2,0);
            \node at (0.75, 1) {$3$};

            \draw (0,2)--(4,0);
            \node at (2.45, 1) {$2$};

           \draw[fill=black]  (0,2) circle (2pt);
            \node at (0, 2.3) {$2$};
           \draw[fill=black]  (4,2) circle (2pt);
           \node at (4, 2.3) {$\infty$};
            \draw[fill=black]  (8,2) circle (2pt);
           \node at (8, 2.3) {$\infty$};

            \draw[fill=black]  (6,0) circle (2pt);
           \node at (6, -0.3) {$2$};

            \draw[fill=black]  (8,0) circle (2pt);
           \node at (8, -0.3) {$3$};

           \draw (6,0)--(8,0);
           \draw (6,0)--(8,2);
           \draw (8,0)--(8,2);
            \node at (7, 0.2) {$2$};
           \node at (8.2, 1) {$2$};
            \node at (7, 1.3) {$2$};
\end{tikzpicture}
        \caption{A Dyer graph $\Gamma$ and the corresponding even Dyer graph $\Omega$.}
	      \label{fig:GammaandOmega}
    \end{center}
\end{figure}

We define a map $\psi\colon V(\Gamma)\to D_\Omega$ as follows: Let $v\in V(\Gamma)$, then there exists $i\in \left\{1,\ldots, k\right\}$ such that $v\in V_i$ and we define $\psi(v):=w_i$. It is straightforward to verify that the map $\psi$ induces an epimorphism
$$\Psi\colon D_\Gamma\twoheadrightarrow D_\Omega.$$

\end{enumerate}

\begin{thm}
\label{QuasiPerfectDyerGroups}
Let $(\Gamma, m, f)$ be a Dyer graph and $D_\Gamma$ be the associated Dyer group. Let $V_1,\ldots, V_k$ be the vertex sets of the connected components of $\Gamma^2$ and $\Delta_1,\ldots,\Delta_k\subseteq\Gamma$ be the induced subgraphs. 

Then $D_\Gamma$ is quasi-perfect if and only if the following hold:
\begin{enumerate}
\item[(i)] For $i\in\left\{1,\ldots, k\right\}$ and for every prime $p$ the graph $\Delta^p_i$ is connected.
\item[(ii)] For each pair $(i,j)$, $1\leq i<j\leq k$ there exist vertices $v\in V_i$ and $w\in V_j$ such that $m(\left\{v,w\right\})=2$. 
\end{enumerate}
\end{thm}
\begin{proof}
Let $\Omega$ be the even Dyer graph that is associated to $\Gamma$ as described in (3).

Let us first assume that $D_\Gamma$ is quasi-perfect. Using the retractions in (2) we obtain $1={\rm dl}(D_\Gamma)\geq {\rm dl}(D_{\Delta_i})$. Hence, $D_{\Delta_i}$ is quasi-perfect for $i=1,\ldots, k$. As we mentioned before,  $D_{\Delta_i}$ is cyclic or is a Coxeter group.
By \cite[Proposition 5.1]{BrooksbankPiggott2012} follows that
for every prime number $p$ the graph  $\Delta_i^p$  is connected.

Now we consider the epimorphism 
$\Psi\colon D_\Gamma\twoheadrightarrow D_\Omega$.
We have
$$1={\rm dl}(D_\Gamma)\geq {\rm dl}(D_\Omega).$$
Hence, ${\rm dl}(D_\Omega)=1$. Since $D_\Omega$ is even, we can apply Proposition \ref{evenDyerGroups} that  shows that $D_\Omega$ is abelian. In particular, for each pair $(i,j)\in\left\{1,\ldots, k\right\}\times\left\{1,\ldots, k\right\}$, $1\leq i<j\leq k$ the label $a_{i,j}=2$. Thus, for each pair $(i,j)\in\left\{1,\ldots, k\right\}\times\left\{1,\ldots, k\right\}$, $1\leq i<j\leq k$ there exist vertices $v\in V_i$ and $w\in V_j$ such that $m(\left\{v,w\right\})=2$.

Conversely, assume that the conditions (i) and (ii) hold for $D_\Gamma$.
Then the associated even Dyer group $D_\Omega$ is abelian. 

For $i\in\left\{1,\ldots, k\right\}$ the subgroup $V_{\Delta_i}$ is cyclic and is therefore quasi-perfect or is a Coxeter group and is quasi-perfect by \cite[Proposition 5.1]{BrooksbankPiggott2012}. For $x_i\in V_i$ we know that $x_i$ and $w_i$ are conjugate, therefore
$$x_iD^{(1)}_{V_i}=w_iD^{(1)}_{V_i}.$$
Since $D^{(1)}_{V_i}=D^{(2)}_{V_i}$ we have
$x_iD^{(2)}_{V_i}=w_iD^{(2)}_{V_i}$.

Let $G$ be a group. For every subgroup $H\subseteq G$ we have $H^{(i)}\subseteq G^{(i)}$ for all $i\geq 0$, see \cite[Theorem 5.15]{Rotman1995}.
Thus $D^{(2)}_{V_i}\subseteq D^{(2)}_\Gamma$ and therefore 
$$x_iD^{(2)}_\Gamma=w_iD^{(2)}_\Gamma.$$

Our goal is to show that $D_\Gamma/D^{(2)}_\Gamma$ is abelian.
Let $gD^{(2)}_\Gamma$ be in $D_\Gamma/D^{(2)}_\Gamma$. For $i\in\left\{1,\ldots, k\right\}$ and $v\in V_i$ we may replace each occurrence of $v$ in $gD^{(2)}_\Gamma$ by $w_iD^{(2)}_\Gamma$. This observation shows that
$$gD^{(2)}_\Gamma\in\langle w_1D^{(2)}_\Gamma,\ldots, w_kD^{(2)}_\Gamma\rangle.$$

The group $\langle w_1D^{(2)}_\Gamma,\ldots, w_kD^{(2)}_\Gamma\rangle$ is abelian. More precisely: let $i,j\in\left\{1,\ldots, k\right\}$, $i\neq j$. By assumption (ii) there exist vertices $v\in V_i$ and $w\in V_j$ such that $vw=wv$. Hence
$$w_iw_jD^{(2)}_\Gamma=vwD^{(2)}_\Gamma=wvD^{(2)}_\Gamma=w_jw_iD^{(2)}_\Gamma.$$

In particular, this shows that $D_\Gamma/D^{(2)}_\Gamma$ is abelian and therefore we obtain $D^{(1)}_\Gamma=D^{(2)}_\Gamma$.

\end{proof}

A group $G$ is called \emph{virtually free} if it contains a free subgroup of finite index. 
A  graph is called \emph{chordal} if every induced circle of length $\geq 4$ has a chord, that means induced circles have length at most $3$.
Using Bass-Serre theory it was proven in \cite{Varghese2026} that virtual freeness of Dyer groups can also be characterized using Dyer graphs.

\begin{thm}(\cite{Varghese2026})
\label{virtuallyfree}
Let $(\Gamma, m, f)$ be a Dyer graph and $D_\Gamma$ be the associated Dyer group. Then $D_\Gamma$ is virtually free if and only if:
\begin{enumerate}
\item if $f(v)=f(w)=\infty$ and $v\neq w$, then $\left\{v,w\right\}\notin E(\Gamma)$,
\item if $f(u)=\infty$, $f(v), f(w)\in\mathbb{N}_{\geq 2}$ such that $v\neq w$, and $\left\{u,v\right\}, \left\{u, w\right\}\in E(\Gamma)$, then $\left\{v,w\right\}\in E(\Gamma)$, 
\item $\Gamma$ is chordal, and
\item if $\Omega\subseteq \Gamma$ is a complete subgraph such that $D_\Omega$ is a Coxeter group, then $D_\Omega$ is finite.
\end{enumerate}
\end{thm}

It is known that the derived length of a non-abelian free group is infinite, but there are many quasi-perfect virtually free Dyer groups. For example, the associated Dyer group $D_\Gamma$, where $\Gamma$ is as shown in \Cref{fig:QuasiPerfectVirtuallyFree}, is quasi-perfect and virtually free.

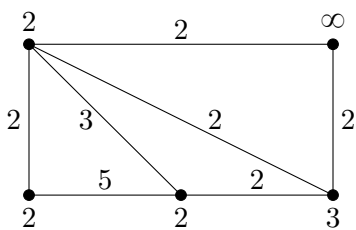
\begin{figure}[htb]
	\begin{center}
	\captionsetup{justification=centering}
		\begin{tikzpicture}
			\draw[fill=black]  (0,0) circle (2pt);
            \node at (0, -0.3) {$2$};
            \draw[fill=black]  (2,0) circle (2pt);
            \node at (2, -0.3) {$2$};
            \draw[fill=black]  (4,0) circle (2pt);
            \node at (4, -0.3) {$3$};
            \draw (0,0)--(2,0);
            \node at (1, 0.2) {$5$};
            \draw (0,0)--(0,2);
            \node at (-0.2, 1) {$2$};

            \draw (0,2)--(4,2);
            \node at (2, 2.2) {$2$};

            \draw (4,0)--(4,2);
            \node at (4.2, 1) {$2$};

            \draw (2,0)--(4,0);
            \node at (3, 0.2) {$2$};

            \draw (0,2)--(2,0);
            \node at (0.75, 1) {$3$};

            \draw (0,2)--(4,0);
            \node at (2.45, 1) {$2$};

           \draw[fill=black]  (0,2) circle (2pt);
            \node at (0, 2.3) {$2$};
           \draw[fill=black]  (4,2) circle (2pt);
           \node at (4, 2.3) {$\infty$};

\end{tikzpicture}
        \caption{A Dyer graph $\Gamma$.}
	      \label{fig:QuasiPerfectVirtuallyFree}
    \end{center}
\end{figure}

\bibliographystyle{halpha}
\bibliography{refs.bib}

@article{Baudisch1981,
Author={Baudisch, A.},
Title = {Subgroups of semifree groups}, 
Journal = {Acta Mathematica Academiae Scientiarum Hungarica},
Year = {1981},
Volume = {38},
Number={1},
Pages = {19--28},
}

@phdthesis{Green1990,
    AUTHOR = {Green, E. R.},
    TITLE = {Graph products of groups},
    YEAR = {1990},
    SCHOOL = {University of Leeds},
    TYPE = {Ph{D} {T}hesis}
}

@article{Dyer1990,
Author ={Dyer,M.},
Title={Reflection subgroups of {C}oxeter systems},
Journal={Journal of Algebra},
Volume={135},
Year={1990}, 
Number={1}, 
Pages={57--73},
}

@article {BrooksbankPiggott2012,
    AUTHOR = {Brooksbank, P. A. and Piggott, A.},
     TITLE = {On the derived length of {C}oxeter groups},
   JOURNAL = {Communications in Algebra},
    VOLUME = {40},
      YEAR = {2012},
     PAGES = {1142--1150},
}

@article{Jeong2006,
title= {On the Derived Series of 3- and 4-Generator {C}oxeter Groups}, 
Author={Jeong, K.-W.},
Journal={Algebra Colloquium},
Year={2006},
Volume={13},
Number={4},
Pages={559--568},
}

@article{EdjvetJeong2007,
Author={Edjvet, M. and Jeong, K.-W.},
Year={2007}, 
Title={On the Derived Series of {C}oxeter Groups},
Journal={Communications in Algebra},
Volume={35},
Number={12}, 
Pages={3937--3947},
}

@book{Rotman1995,
Author={Rotman, J. J.},
Title= {An Introduction to the Theory of Groups},
Year= {1995},
 PUBLISHER = {4th Edition, Springer-Verlag, New York},
}

@book{LyndonSchupp2001,
author={R. C. Lyndon and P. E. Schupp},
Title= {Combinatorial Group Theory},
Publisher={Springer-Verlag, Berlin, Heidelberg, New York},
Year= {1977},
}

@article{Varghese2026,
    author = {Varghese, O.},
    title = {Algebraic, geometric and profinite properties of {D}yer groups},
    journal = {In preparation},
    year = {2026}
}

@article{MollerVarghese2023,
      title={On quotients of {C}oxeter groups}, 
      author={Philip M\"{o}ller and Olga Varghese},
     JOURNAL = {J. Algebra},
  FJOURNAL = {Journal of Algebra},
  year={2024},
  volume={639},
number = {},
  pages={516--531},
}

@book {Serre2003,
    AUTHOR = {Serre, Jean-Pierre},
     TITLE = {Trees},
    SERIES = {Springer Monographs in Mathematics},
      NOTE = {Translated from the French original by John Stillwell,
              Corrected 2nd printing of the 1980 English translation},
 PUBLISHER = {Springer-Verlag, Berlin},
      YEAR = {2003},
     PAGES = {x+142},
      ISBN = {3-540-44237-5},
   MRCLASS = {20E08 (05C05 20E06 20G25)},
  MRNUMBER = {1954121},
}
\end{document}